\documentclass[10pt]{amsart}

\usepackage{graphicx}
\usepackage{amscd}
\usepackage{amsmath}
\usepackage{amsfonts}
\usepackage{amssymb}
\usepackage{amsthm}
\usepackage{enumerate}

\newcommand{\field}{\mathcal{N}}
\newcommand{\setR}{\ensuremath{\mathbb{R}}}

\newtheorem{theorem}{Theorem}

\theoremstyle{definition}
\newtheorem{defn}{Definition}
\newtheorem{notation}{Notation}

\newcommand{\abs}[1]{\left\lvert#1\right\rvert}

\begin{document}
\title[Analyticity of WLUD$^\infty$ functions on $\mathcal{N}$ and $\mathcal{N}^n$]{On the analyticity of WLUD$^\infty$ functions of one variable and WLUD$^\infty$ functions of several variables in a complete non-Archimedean valued field}
\author{Khodr Shamseddine}
\address{Department of Physics and Astronomy, University of Manitoba, Winnipeg, Manitoba
R3T 2N2, Canada}
\email{khodr.shamseddine@umanitoba.ca}
\thanks{This research was funded by the Natural Sciences and Engineering Council of Canada (NSERC, Grant \# RGPIN/4965-2017)}
\subjclass[2020]{41A58, 32P05, 12J25, 26E20, 46S10}
\keywords{Taylor Series Expansion; Analytic Functions; non-Archimedean Analysis; non-Archimedean Valued Fields}
\begin{abstract} Let $\mathcal{N}$ be a non-Archimedean ordered field extension of the real numbers that is real closed and Cauchy complete in the topology
induced by the order, and whose Hahn group is Archimedean. In this paper, we first review the properties of weakly locally uniformly differentiable
(WLUD) functions, $k$ times weakly locally uniformly differentiable (WLUD$^k$) functions, and WLUD$^\infty$ functions \cite{boo-sham-18,MultivWLUD} at a point or on an open subset of $\mathcal{N}$. Then we show under what conditions a WLUD$^\infty$ function at a point $x_0\in\mathcal{N}$ is analytic in an interval around $x_0$, that is, it has a convergent Taylor series at any point in that interval.

We generalize the concepts of WLUD$^k$ and WLUD$^\infty$ to functions from $\mathcal{N}^n$ to $\mathcal{N}$, for any
$n\in\mathbb{N}$. Then we formulate conditions under which a WLUD$^\infty$ function at a point $\boldsymbol{x_0} \in \mathcal{N}^n$ is analytic at that point.
\end{abstract}
\maketitle

\section{Introduction} \label{seintro}

Let $\mathcal{N}$ be a non-Archimedean ordered field extension of $\ensuremath{\mathbb{R}}$ that is real closed and
complete in the order topology and whose Hahn group $S_\mathcal{N}$ is Archimedean, i.e. (isomorphic to) a subgroup of $\ensuremath{\mathbb{R}}$. Recall that $S_{\mathcal{N}}$ is the set of equivalence classes under the relation $\sim$ defined on $\mathcal{N}^*:=\mathcal{N}\setminus\{0\}$ as follows: For $x,y\in \mathcal{N}^*$, we say that $x$ is of the same order as $y$ and write $x\sim y$ if there exist
$n,m\in\ensuremath{\mathbb{N}}$
such that $n|x|>|y|$ and $m|y|>|x|$, where $|\cdot|$ denotes the ordinary absolute value on $\mathcal{N}$:
$ |x|=\max\left\{x,
-x\right \}$.
$S_{\mathcal{N}}$ is naturally endowed with an addition via $[x]+[y]=[x\cdot y]$ and an order via
$[x]<[y]$ if $|y|\ll|x|$ (which means $n|y|< |x|$ for all $n\in\mathbb{N}$), both of which are readily checked to be well-defined.
It follows that $(S_{\mathcal{N}},+,<)$ is an ordered group,
often referred to as the Hahn group or skeleton group, whose neutral element is $[1]$, the class of $1$.

The theorem of Hahn \cite{hahn} provides a complete classification
of non-Archimedean ordered field extensions of $\ensuremath{\mathbb{R}}$ in terms
of their skeleton groups. In fact, invoking the axiom of choice it
is shown that the elements of our field $\mathcal{N}$ can be written as
(generalized) formal power series (also called Hahn series) over its skeleton group $S_{\mathcal{N}}$ with real
coefficients, and the set of appearing exponents forms a
well-ordered subset of $S_{\mathcal{N}}$. That is, for all $x\in \mathcal{N}$, we have that
$x=\sum_{q\in S_{\mathcal{N}}}a_qd^{q}$;
with $a_q\in\setR$ for all $q$, $d$ a positive infinitely small element of $\mathcal{N}$, and the support of $x$, given by
$\mbox{supp}(x):=\{q\in S_\mathcal{N}: a_q\ne 0\}$,
forming a well-ordered subset of $S_{\mathcal{N}}$.

 We define for $x\ne 0$ in $\mathcal{N}$,
$ \lambda(x)=\min\left(\mbox{supp}(x)\right)$,
  which exists since $\mbox{supp}(x)$ is well-ordered. Moreover, we set $\lambda(0)=\infty$. Given a nonzero $x=\sum_{q\in \mbox{supp}(x)}a_qd^q$, then $x>0$ if and only if $a_{\lambda(x)}>0$.

The smallest such
field $\mathcal{N}$ is the Levi-Civita field $\mathcal{R}$, first introduced in \cite{levicivita1,levicivita2}. In this case
$S_\mathcal{R}=\ensuremath{\mathbb{Q}}$, and for any element $x\in\mathcal{R}$, supp$(x)$ is a left-finite
subset of $\ensuremath{\mathbb{Q}}$, i.e. below any rational bound $r$ there are only finitely many exponents in the Hahn representation of $x$. The Levi-Civita field
$\mathcal{R}$ is of particular interest because of its practical usefulness. Since the supports of the elements of $\mathcal{R}$ are left-finite, it is possible to represent these
numbers on a computer. Having infinitely
small numbers allows for many computational
applications; one
such application is the computation of derivatives of real functions representable on a computer \cite{rsdiffsf,dabul00}, where both the accuracy of
formula manipulators and the speed of classical numerical methods are achieved. For a review of the Levi-Civita field $\mathcal{R}$, see
\cite{rsrevitaly13} and references therein.

In the wider context of valuation theory, it is
interesting to note that the topology induced by the order on $\mathcal{N}$ is the same as the valuation topology $\tau_v$ introduced via the ultrametric $\Lambda:\mathcal{N}\times\mathcal{N}\rightarrow\mathbb{R}$, given
by $\Lambda(x,y)=\exp{(-\lambda(x-y))}$. It follows therefore that the field $\mathcal{N}$ is just a special
case of the class of fields discussed in \cite{schikhofbook}. For a general
overview of the algebraic properties of formal power series fields, we refer to the comprehensive overview by Ribenboim \cite{ribenboim92}, and
for an overview of the related valuation theory the book by Krull \cite{krull32}. A thorough and complete treatment of ordered structures can
also be found in \cite{priessbook}. A more comprehensive survey of all non-Archimedean fields can be found in \cite{barria-sham-18}.

\section{Weak Local Uniform Differentiability and Review of Recent Results}

Because of the total disconnectedness of the field $\mathcal{N}$ in the order topology, the standard theorems of real calculus like the
intermediate value theorem, the inverse function theorem, the mean value theorem, the implicit function theorem and Taylor's theorem require stronger smoothness criteria of the functions involved in order for the theorems to hold.
In this section we will present one such criterion: the so-called \lq weak local uniform differentiability\rq,
we will review recent work based on that smoothness criterion and then present new results.

In \cite{boo-sham-18}, we focus our attention on $\mathcal{N}$-valued functions of one variable. We study the properties of weakly locally uniformly differentiable (WLUD)
functions at a point $x_0\in\mathcal{N}$ or on an open subset $A$ of $\mathcal{N}$. In particular, we show that WLUD functions are $C^1$, they include all polynomial functions,
and they are closed under addition, multiplication and composition. Then we generalize the definition of weak local uniform differentiability to any order. In particular,
we study the properties of WLUD$^2$ functions at a point $x_0\in\mathcal{N}$ or on an open subset $A$ of $\mathcal{N}$; and we show that WLUD$^2$ functions are $C^2$,
they include all polynomial functions, and they are closed under addition, multiplication and composition. Finally, we
formulate and prove an inverse function theorem as well as a local intermediate value theorem and a local mean value theorem for these functions.

Here we only recall the main definitions and results (without proofs) in \cite{boo-sham-18} and refer the reader to that paper for the details.

\begin{defn}
Let $A\subseteq \mathcal{N}$ be open, let $f:A\rightarrow \mathcal{N}$, and let $x_0\in A$ be given. We say that $f$ is weakly locally uniformly differentiable (abbreviated as WLUD)
at $x_0$ if $f$ is differentiable in a neighbourhood $\Omega$ of $x_0$ in $A$ and if for every $\epsilon > 0$ in $\mathcal{N}$ there exists $\delta > 0$ in $\mathcal{N}$ such that $(x_0 - \delta, x_0 + \delta) \subset \Omega$, and
for
every $x,y \in (x_0 - \delta, x_0 + \delta)$ we have that $\abs{f(y) - f(x) - f^\prime (x)(y-x)} \le \epsilon \abs{y-x}$. Moreover, we say that $f$ is WLUD on $A$ if $f$
is WLUD at every point in $A$.
\end{defn}

We extend the WLUD concept to higher orders of differentiability and we define WLUD$^k$ as follows.

\begin{defn}\label{def:wludn}
Let $A\subseteq \mathcal{N}$ be open, let $f:A\rightarrow \mathcal{N}$, let $x_0\in A$, and let $k\in\mathbb{N}$ be given. We say that $f$ is WLUD$^k$ at $x_0$ if $f$ is $k$ times
differentiable in a neighbourhood $\Omega$ of $x_0$ in $A$ and if for every $\epsilon > 0$ in $\mathcal{N}$ there exists $\delta > 0$ in $\mathcal{N}$ such that $(x_0-\delta,x_0+\delta) \subset \Omega$, and for every
$x,y \in (x_0-\delta,x_0+\delta)$  we have that
\[
\left|f(y) - \sum\limits_{j=0}^k \frac{f^{(j)}(x)}{j!}(y-x)^j\right|\le \epsilon \left|y-x\right|^k.
\]
 Moreover, we say that $f$ is WLUD$^k$  on $A$ if $f$ is WLUD$^k$ at every point in $A$. Finally, we say that $f$ is  WLUD$^\infty$ at $x_0$ (respectively, on $A$) if $f$ is WLUD$^k$ at $x_0$ (respectively, on $A$) for every
$k\in\mathbb{N}$.
\end{defn}

\begin{theorem}[Inverse Function Theorem]
Let $A\subseteq\mathcal{N}$ be open, let $f:A\rightarrow \field $ be WLUD on $A$, and let $x_0 \in A$ be such that $f^\prime(x_0) \neq 0$. Then there exists a neighborhood $\Omega$
of $x_0$
in $A$ such that
\begin{enumerate}
\item $\left.f\right|_\Omega$ is one-to-one;
\item $f(\Omega)$ is open; and
\item $f^{-1}$ exists and is WLUD on $f(\Omega)$ with $(f^{-1})^\prime = 1/\left(f^\prime \circ f^{-1}\right)$.
\end{enumerate}
\end{theorem}

\begin{theorem}[Local Intermediate Value Theorem]\label{tlivt}
Let $A\subseteq\mathcal{N}$ be open, let $f:A\rightarrow \field $ be WLUD on $A$, and let $x_0 \in A$ be such that $f^\prime(x_0) \neq 0$.  Then there exists a neighborhood $\Omega$
of $x_0$ in $A$ such that for
 any $a<b$ in $f(\Omega)$ and for any $c\in (a,b)$, there is an $x\in\left(\min\left\{f^{(-1)}(a),f^{(-1)}(b)\right\},\max\left\{f^{(-1)}(a),f^{(-1)}(b)\right\}\right)$
 such that $f(x)=c$.
\end{theorem}

\begin{theorem}[Local Mean Value Theorem]\label{thmvt}
Let $A\subseteq\mathcal{N}$ be open, let $f:A\rightarrow \field $ be WLUD$^2$ on $A$, and let $x_0 \in A$ be such that  $f^{\prime\prime}(x_0) \neq 0$. Then there exists a neighborhood $\Omega$ of
$x_0$ in $A$ such that $f$ has the mean value property on $\Omega$. That is, for every $a,b\in \Omega$ with $a<b$, there exists $c\in (a,b)$ such that
\[
f^\prime(c) = \frac{f(b) - f(a)}{b-a}.
\]
\end{theorem}

As in the real case, the mean value property can be used to prove other important results. In particular, while  L'H\^opital's rule does not hold for differentiable functions on
$\mathcal{N}$, we prove the result under similar conditions to those of the local mean value theorem.

\begin{theorem}[L'H\^opital's Rule]
Let $A\subset \mathcal{N}$ be open, let $f,g:A\rightarrow \mathcal{N}$ be WLUD$^2$ on $A$, and let $a\in A$ be such that $f^{\prime\prime}(a) \neq 0$ and $g^{\prime\prime}(a) \neq 0$.
Furthermore, suppose that $f(a) = g(a) = 0$, that there exists a neighborhood
$\Omega$ of $a$ in $A$ such that $g^{\prime}(x) \neq 0$ for every $x\in \Omega\setminus\{a\}$, and that  $\lim\limits_{x\rightarrow a} f^{\prime}(x)/g^{\prime}(x)$ exists. Then
\[
\lim_{x\rightarrow a} \frac{f(x)}{g(x)} = \lim_{x\rightarrow a} \frac{f^{\prime}(x)}{g^{\prime}(x)}.
\]
\end{theorem}

In \cite{MultivWLUD},  we formulate and prove a Taylor theorem with remainder  for WLUD$^k$ functions from $\mathcal{N}$ to $\mathcal{N}$. Then we extend the concept of WLUD to functions from  $\mathcal{N}^n$ to
$\mathcal{N}^m$ with $m,n\in\mathbb{N}$ and study the properties of those functions as we did for functions from $\mathcal{N}$ to $\mathcal{N}$. Then we formulate and prove
the inverse function theorem for WLUD functions from $\mathcal{N}^n$ to
$\mathcal{N}^n$ and the implicit function theorem for WLUD functions from $\mathcal{N}^n$ to $\mathcal{N}^m$ with $m<n$ in $\mathbb{N}$.

As in the real case, the proof of Taylor's theorem with remainder uses the mean value theorem.
However, in the non-Archimedean setting, stronger conditions on the function are needed than in the real case for the formulation of the theorem.

\begin{theorem}\label{thmtaylor}
(Taylor's Theorem with Remainder) Let $A \subseteq \mathcal{N}$ be open, let $k\in\mathbb{N}$ be given, and let $f : A \rightarrow \mathcal{N}$ be WLUD$^{k+2}$ on $A$.
Assume further that $f^{(j)}$ is WLUD$^{2}$ on $A$ for $0\leq j \leq k$. Then, for every $x\in A$, there exists a neighborhood $U$ of $x$ in $A$ such that,  for any $y \in U$,
there exists
$c \in  \left[ \min(y, x),\max(y, x) \right]$ such that
\begin{equation}\label{eqtaylor}
f(y)=\sum_{j=0}^{k} \frac{f^{(j)}\left(x\right)}{j !}\left(y-x\right)^{j}+\frac{f^{(k+1)}(c)}{(k+1) !}\left(y-x\right)^{k+1}.
\end{equation}

\end{theorem}

Before we define weak local uniform differentiability for functions from $\mathcal{N}^n$ to $\mathcal{N}^m$ and then state the inverse function theorem and the implicit function theorem, we introduce the following notations.

\begin{notation} Let $A\subset\mathcal{N}^n$ be open, let $\boldsymbol{x_0}\in A$ be given, and let $\boldsymbol{f}:A\rightarrow\mathcal{N}^m$ be such that all the first order partial derivatives of $\boldsymbol{f}$ at  $\boldsymbol{x_0}$ exist. Then $\boldsymbol{D}\boldsymbol{f}(\boldsymbol{x_0})$ denotes the linear map from $\mathcal{N}^n$ to $\mathcal{N}^m$ defined by the $m\times n$
Jacobian matrix of $\boldsymbol{f}$ at $\boldsymbol{x_0}$:
\[
\begin{pmatrix} \boldsymbol{f}^1_1(\boldsymbol{x_0}) & \boldsymbol{f}^1_2(\boldsymbol{x_0})& \ldots & \boldsymbol{f}^1_n(\boldsymbol{x_0})
\\ \boldsymbol{f}^2_1(\boldsymbol{x_0}) &
\boldsymbol{f}^2_2(\boldsymbol{x_0}) & \ldots & \boldsymbol{f}^2_n(\boldsymbol{x_0}) \\ \vdots & \vdots &\ddots & \vdots \\
\boldsymbol{f}^m_1(\boldsymbol{x_0}) & \boldsymbol{f}^m_2(\boldsymbol{x_0}) & \ldots & \boldsymbol{f}^m_n(\boldsymbol{x_0}) \end{pmatrix}
\]
with $\boldsymbol{f}^i_j(\boldsymbol{x_0})=\frac{\partial f_i}{\partial x_j}(\boldsymbol{x_0})$ for $1\le i\le m$ and $1\le j\le n$. Moreover, if $m=n$ then the determinant of the $n\times n$ matrix $\boldsymbol{D}\boldsymbol{f}(\boldsymbol{x_0})$ is denoted by $J\boldsymbol{f}(\boldsymbol{x_0})$.
\end{notation}
\begin{defn}[WLUD]\label{defWLUDnm}
Let $A\subset\mathcal{N}^n$ be open, let $\boldsymbol{f}:A \to \mathcal{N}^m$, and let $\boldsymbol{x_0}\in A$ be given.  Then we say that $\boldsymbol{f}$ is weakly
locally uniformly differentiable (WLUD) at $\boldsymbol{x_0}$ if $\boldsymbol{f}$ is differentiable in a neighborhood
$\Omega$ of $\boldsymbol{x_0}$ in $A$ and if for every $\epsilon>0$ in $\mathcal{N}$ there exists $\delta>0$ in $\mathcal{N}$ such that $B_{\delta}(\boldsymbol{x_0}):=\left\{\boldsymbol{t}\in\mathcal{N}:\left|\boldsymbol{t}-\boldsymbol{x_0}\right|<\delta\right\}\subset\Omega$, and
for all $\boldsymbol{x},\boldsymbol{y}\in B_{\delta}(\boldsymbol{x_0})$ we have that
\[
\left|\boldsymbol{f}(\boldsymbol{y}) - \boldsymbol{f}(\boldsymbol{x}) -
\boldsymbol{D}\boldsymbol{f}(\boldsymbol{x})(\boldsymbol{y} - \boldsymbol{x})\right| \le \epsilon \vert \boldsymbol{y} - \boldsymbol{x} \vert.
\]
Moreover, we say that $\boldsymbol{f}$ is WLUD on $A$ if $\boldsymbol{f}$ is WLUD at every point in $A$.
\end{defn}

We show in \cite{MultivWLUD} that if $\boldsymbol{f}$ is WLUD at $\boldsymbol{x_0}$ (respectively on $A$) then $\boldsymbol{f}$ is C$^1$ at $\boldsymbol{x_0}$ (respectively on $A$). Thus, the class of WLUD functions at a point $\boldsymbol{x_0}$ (respectively on an open set $A$) is a subset of
the class of
$C^1$ functions at $\boldsymbol{x_0}$ (respectively on $A$). However, this is still large enough to include all polynomial functions. We also show in \cite{MultivWLUD} that if $\boldsymbol{f},\boldsymbol{g}$ are WLUD at $\boldsymbol{x_0}$ (respectively on $A$) and if $\alpha\in\mathcal{N}$ then $\boldsymbol{f}+\alpha\boldsymbol{g}$ and $\boldsymbol{f}\cdot\boldsymbol{g}$ are WLUD at $\boldsymbol{x_0}$ (respectively on $A$). Moreover, we show that if $\boldsymbol{f}:A \to \mathcal{N}^m$ is WLUD at $\boldsymbol{x_0}\in A$ (respectively on $A$) and if $\boldsymbol{g}:C \to \mathcal{N}^p$ is WLUD at $\boldsymbol{f}(\boldsymbol{x_0})\in C$ (respectively on $C$),
where $A$ is an open subset of $\mathcal{N}^n$, $C$ an open subset of $\mathcal{N}^m$ and $\boldsymbol{f}(A) \subseteq C$, then
$\boldsymbol{g} \circ \boldsymbol{f}$ is WLUD at $\boldsymbol{x_0}$ (respectively on $A$).

\begin{theorem}[Inverse Function Theorem]\label{IFT}
Let $A\subset\mathcal{N}$ be open, let $\boldsymbol{g}:A\rightarrow\mathcal{N}^n$ be WLUD on $A$ and let $\boldsymbol{t_0}\in A$ be such that
$J\boldsymbol{g}(\boldsymbol{t_0})\neq0$. Then there is a neighborhood $\Omega$ of $\boldsymbol{t_0}$ such that:
\begin{enumerate}
\item $\boldsymbol{g}|_\Omega$ is one-to-one;
\item $\boldsymbol{g}(\Omega)$ is open;
\item the inverse $\boldsymbol{f}$ of $\boldsymbol{g}|_\Omega$ is WLUD on $\boldsymbol{g}(\Omega)$; and
$\boldsymbol{D}\boldsymbol{f}(\boldsymbol{x})=\left[\boldsymbol{D}\boldsymbol{g}(\boldsymbol{t})\right]^{-1}$ for $\boldsymbol{t}\in\Omega$ and
$\boldsymbol{x}=\boldsymbol{g}(\boldsymbol{t})$.
\end{enumerate}
\end{theorem}

As in the real case, the inverse function theorem is used to prove the implicit function theorem. But before we state the implicit function theorem, we introduce the following notations.
\begin{notation}Let $A\subseteq\mathcal{N}^n$ be open and let $\boldsymbol{\Phi}: A\rightarrow\mathcal{N}^m$ be WLUD on $A$. For $\boldsymbol{t}=(t_1,...,t_{n-m},t_{n-m+1},...,t_{n} )\in A$, let
\begin{equation*}
\hat{\boldsymbol{t}}=(t_1,...,t_{n-m})\text{ and }\tilde{J}\boldsymbol{\Phi}(\boldsymbol{t})=\det\left(\dfrac{\partial(\Phi_1,...,\Phi_m)}{\partial(t_{n-m+1},...,t_{n})}\right).
\end{equation*}
\end{notation}

\begin{theorem}[Implicit Function Theorem]
Let $\boldsymbol{\Phi}:A\rightarrow\mathcal{N}^m$ be WLUD on $A$, where $A\subseteq\mathcal{N}^n$ is open and $1\leq m<n.$
Let $\boldsymbol{t_0}\in A$ be such that $\boldsymbol{\Phi}(\boldsymbol{t_0})=\boldsymbol{0}$ and
$\tilde{J}\boldsymbol{\Phi}(\boldsymbol{t_0})\neq0$. Then there exist a neighborhood $U$ of $\boldsymbol{t_0}$, a neighborhood $R$ of
$\hat{\boldsymbol{t_0}}$ and $\boldsymbol{\phi}:R\rightarrow\mathcal{N}^m$ that is WLUD on $R$ such that
\[
\tilde{J}\boldsymbol{\Phi}(\boldsymbol{t})\neq0
 \text{ for all } \boldsymbol{t}\in U,
 \]
 and
\begin{equation*}
\{\boldsymbol{t}\in
U:\boldsymbol{\Phi}(\boldsymbol{t})=\boldsymbol{0}\}=\{(\hat{\boldsymbol{t}},\boldsymbol{\phi}(\hat{\boldsymbol{t}})):\hat{\boldsymbol{t}}\in
R\}.
\end{equation*}
\end{theorem}

\section{New Results}
This paper is a continuation of the work done in
\cite{boo-sham-18,MultivWLUD}. In the following section, we will generalize in Definition \ref{defWLUDn1k} and Definition \ref{defWLUDn1infty} the concepts of WLUD$^k$ and WLUD$^\infty$ to functions from $\mathcal{N}^n$ to $\mathcal{N}$; and we will formulate (in Theorem \ref{thmtaylorseries1} and Theorem \ref{thmtaylorseriesn} and their proofs) conditions under which a WLUD$^\infty$ $\mathcal{N}$-valued function at a point $x_0\in \mathcal{N}$ or a WLUD$^\infty$ $\mathcal{N}$-valued function at a point $\boldsymbol{x_0} \in \mathcal{N}^n$ will be analytic at that point.

\begin{theorem}\label{thmtaylorseries1}
Let $A \subseteq \mathcal{N}$ be open, let $x_0\in A$, and let $f : A \rightarrow \mathcal{N}$ be WLUD$^{\infty}$ at $x_0$. For each $k\in\mathbb{N}$, let $\delta_k>0$ in $\mathcal{N}$ correspond to $\epsilon=1$ in Definition \ref{def:wludn}. Assume that
\[
\limsup_{j\rightarrow\infty}\left(\frac{-\lambda\left(f^{(j)}(x_0)\right)}{j}\right)<\infty \text{ and }
\limsup _{k\rightarrow\infty}\lambda\left(\delta_k\right)<\infty.
\]
Then there exists a neighborhood $U$ of $x_0$ in $A$ such that,  for any $x,y \in U$, we have that
\[
f(y)=\sum_{j=0}^{\infty} \frac{f^{(j)}\left(x\right)}{j!}\left(y-x\right)^{j}.
\]
That is, the Taylor series $\sum\limits_{j=0}^{\infty} \frac{f^{(j)}\left(x\right)}{j!}\left(y-x\right)^{j}$ converges in $\mathcal{N}$ to $f(y)$; and hence $f$ is analytic in $U$.
\end{theorem}
\begin{proof}
Let
\[
\lambda_0=\limsup_{j\rightarrow\infty}\left(\frac{-\lambda\left(f^{(j)}(x_0)\right)}{j}\right).
\]
Then $\lambda_0\in\mathbb{R}$ and $\lambda_0<\infty$; and, by \cite[Page 59]{schikhofbook}, we have that $\sum\limits_{j=0}^{\infty} \frac{f^{(j)}\left(x_0\right)}{j!}\left(x-x_0\right)^{j}$ converges in $\mathcal{N}$ for all $x\in \mathcal{N}$ satisfying $\lambda(x-x_0)>\lambda_0$.

For all $k\in \mathbb{N}$, we have that $(x_0-\delta_k, x_0+\delta_k)\subset A$, $f$ is $k$ times differentiable on $(x_0-\delta_k, x_0+\delta_k)$, and
 \[
 \left|f(x)-\sum\limits_{j=0}^{k} \frac{f^{(j)}\left(x_0\right)}{j!}\left(x-x_0\right)^{j}\right|\le \left|x-x_0\right|^k \text{ for all }x\in (x_0-\delta_k, x_0+\delta_k).
\]
 Since $\limsup\limits _{k\rightarrow\infty}\lambda\left(\delta_k\right)<\infty$, there exists $t>0$ in $\mathbb{Q}$ such that $\limsup\limits _{k\rightarrow\infty}\lambda\left(\delta_k\right)<t<\infty$. Thus, there exists $N\in\mathbb{N}$ such that
 \begin{equation}\label{eqtaylor1:1}
  \lambda(\delta_k)<t\text{ for all }k>N.
  \end{equation}
  Let
  $\delta>0$ in $\mathcal{N}$ be such that $\lambda(\delta)>\max\{\lambda_0, t,0\}$; this is possible since $\max\{\lambda_0, t,0\}<\infty$. It follows from (\ref{eqtaylor1:1})
   that $\lambda(\delta)>\lambda(\delta_k)$ and hence $0<\delta\ll\delta_k$ for all $k>N$. Thus,
  $(x_0-\delta, x_0+\delta)\subset A$, $f$ is infinitely often differentiable on $(x_0-\delta, x_0+\delta)$, and
 \begin{equation}\label{eqtaylor1:2}
 \left|f(x)-\sum\limits_{j=0}^{k} \frac{f^{(j)}\left(x_0\right)}{j!}\left(x-x_0\right)^{j}\right|\le \left|x-x_0\right|^k \forall\ x\in (x_0-\delta, x_0+\delta)\text{ and }\forall\ k>N.
\end{equation}
Moreover, for all $x\in (x_0-\delta, x_0+\delta)$, we have that $\lambda(x-x_0)\ge \lambda(\delta)>\lambda_0$ and hence $\sum\limits_{j=0}^{\infty} \frac{f^{(j)}\left(x_0\right)}{j!}\left(x-x_0\right)^{j}$ converges in $\mathcal{N}$. Let $U=(x_0-\delta, x_0+\delta)$.

First we show that
\[
f(x)=\sum_{j=0}^{\infty} \frac{f^{(j)}\left(x_0\right)}{j!}\left(x-x_0\right)^{j}\text{ for all }x\in U.
\]
Let $x\in U$ be given. Taking the limit in (\ref{eqtaylor1:2}) as $k\rightarrow\infty$, we get:
\[
0\le\lim_{k\rightarrow\infty}\left|f(x)-\sum\limits_{j=0}^{k} \frac{f^{(j)}\left(x_0\right)}{j!}\left(x-x_0\right)^{j}\right|\le \lim_{k\rightarrow\infty} \left|x-x_0\right|^k,
\]
from which we obtain
\[
0\le\left|f(x)-\lim_{k\rightarrow\infty}\sum\limits_{j=0}^{k} \frac{f^{(j)}\left(x_0\right)}{j!}\left(x-x_0\right)^{j}\right|\le \lim_{k\rightarrow\infty} \left|x-x_0\right|^k.
\]
Since $\lambda(x-x_0)\ge \lambda(\delta)>0$, we obtain that
$\lim\limits_{k\rightarrow\infty} \left|x-x_0\right|^k=0$.
It follows that
\[
0\le\left|f(x)-\sum\limits_{j=0}^{\infty} \frac{f^{(j)}\left(x_0\right)}{j!}\left(x-x_0\right)^{j}\right|\le 0
\]
from which we infer that
$f(x)=\sum\limits_{j=0}^{\infty} \frac{f^{(j)}\left(x_0\right)}{j!}\left(x-x_0\right)^{j}$ or, equivalently,
\begin{equation}\label{eqtaylor1:3}
f(x)=\sum\limits_{l=0}^{\infty} \frac{f^{(l)}\left(x_0\right)}{l!}\left(x-x_0\right)^{l}.
\end{equation}
Since the convergence of the Taylor series above is in the order (valuation) topology, we will show that the derivatives of $f$ at $x$ to any order are obtained by differentiating the power series in Equation (\ref{eqtaylor1:3}) term by term. That is, for all $j\in\mathbb{N}$,
\begin{equation}\label{eqtaylor1:4}
f^{(j)}(x)=\sum_{l=j}^{\infty} l(l-1)\ldots(l-j+1)\frac{f^{(l)}\left(x_0\right)}{l!}\left(x-x_0\right)^{l-j}.
\end{equation}
First note that, since $\lambda\left(l(l-1)\ldots(l-j+1)\right)=0$, it follows that $\sum_{l=j}^{\infty} l(l-1)\ldots(l-j+1)\frac{f^{(l)}\left(x_0\right)}{l!}\left(x-x_0\right)^{l-j}$ converges in $\mathcal{N}$ for all $j\in\mathbb{N}$. Using induction on $j$, it suffices to show that
\[
f^\prime(x)= \sum_{l=1}^{\infty} l\frac{f^{(l)}\left(x_0\right)}{l!}\left(x-x_0\right)^{l-1}=\sum_{l=1}^{\infty} \frac{f^{(l)}\left(x_0\right)}{(l-1)!}\left(x-x_0\right)^{l-1}.
\]
Let $h\in \mathcal{N}$ be such that $x+h\in U$. We will show that
\[
\lim_{h\rightarrow0}\left\{\frac{f(x+h)-f(x)}{h}\right\}=\sum_{l=1}^{\infty} \frac{f^{(l)}\left(x_0\right)}{(l-1)!}\left(x-x_0\right)^{l-1}.
\]
Thus,
\begin{eqnarray*}
&&\lim_{h\rightarrow0}\left\{\frac{f(x+h)-f(x)}{h}\right\}=\lim_{h\rightarrow0}\left\{\sum\limits_{l=0}^{\infty}\frac{f^{(l)}\left(x_0\right)}{l!}\frac{ \left(x+h-x_0\right)^{l}-\left(x-x_0\right)^{l}}{h}\right\}\\
&=&\lim_{h\rightarrow0}\left\{\sum\limits_{l=1}^{\infty}\frac{f^{(l)}\left(x_0\right)}{l!}\frac{ \left(x+h-x_0\right)^{l}-\left(x-x_0\right)^{l}}{h}\right\}\\
&=&\lim_{h\rightarrow0}\left\{\sum\limits_{l=1}^{\infty}\frac{f^{(l)}\left(x_0\right)}{l!} \left[(x+h-x_0)^{l-1}+(x+h-x_0)^{l-2}(x-x_0)+\cdots+(x-x_0)^{l-1}\right]\right\}\\
&=&\sum\limits_{l=1}^{\infty}\frac{f^{(l)}\left(x_0\right)}{l!} \left[l(x-x_0)^{l-1}\right]\\
&=&\sum_{l=1}^{\infty} \frac{f^{(l)}\left(x_0\right)}{(l-1)!}\left(x-x_0\right)^{l-1}.
\end{eqnarray*}

Now let $y\in U$ be given. Then
\begin{eqnarray*}
f(y)&=&\sum_{l=0}^{\infty} \frac{f^{(l)}\left(x_0\right)}{l!}\left(y-x_0\right)^{l}\\
&=&\sum_{l=0}^{\infty} \frac{f^{(l)}\left(x_0\right)}{l!}\left[(y-x)+(x-x_0)\right]^{l}\\
&=&\sum_{l=0}^{\infty} \sum_{j=0}^{l} \frac{f^{(l)}\left(x_0\right)}{l!}\left(\begin{array}{c}l\\j\end{array}\right) (y-x)^j(x-x_0)^{l-j}\\
&=&\sum_{l=0}^{\infty} \sum_{j=0}^{l} \frac{l(l-1)\ldots(l-j+1)}{j!}\frac{f^{(l)}\left(x_0\right)}{l!}(x-x_0)^{l-j}(y-x)^j.
\end{eqnarray*}
Since convergence in the order topology (valuation topology) entails absolute convergence, we can interchange the order of the summations in the last equality \cite{shamseddinephd,rspsio00}. We get:
\begin{eqnarray*}
f(y)&=&\sum_{j=0}^{\infty}\frac{1}{j!}\left[ \sum_{l=j}^{\infty}l(l-1)\ldots(l-j+1)\frac{f^{(l)}\left(x_0\right)}{l!}(x-x_0)^{l-j}\right](y-x)^j\\
&=&\sum_{j=0}^{\infty}\frac{f^{(j)}(x)}{j!}(y-x)^j
\end{eqnarray*}
where we made use of Equation (\ref{eqtaylor1:4}) in the last step.
\end{proof}

Replacing $m$ by $1$ in Definition \ref{defWLUDnm}, then the $m\times n$ matrix $\boldsymbol{D}\boldsymbol{f}(\boldsymbol{x})$ is replaced by the gradient of $\boldsymbol{f}$ at $\boldsymbol{x}$: $\boldsymbol{\nabla}f(\boldsymbol{x})$, and we readily obtain the definition of a WLUD $\mathcal{N}$-valued function at a point $\boldsymbol{x_0}$ or on an open subset $A$ of $\mathcal{N}^n$.

\begin{defn}\label{defWLUDn1}
Let $A\subset\mathcal{N}^n$ be open, let $f:A \to \mathcal{N}$, and let $\boldsymbol{x_0}\in A$ be given.  Then we say that $f$ is WLUD at $\boldsymbol{x_0}$ if $f$ is differentiable in a neighborhood
$\Omega$ of $\boldsymbol{x_0}$ in $A$ and if for every $\epsilon>0$ in $\mathcal{N}$ there exists $\delta>0$ in $\mathcal{N}$ such that $B_{\delta}(\boldsymbol{x_0})\subset\Omega$, and
for all $\boldsymbol{x},\boldsymbol{y}\in B_{\delta}(\boldsymbol{x_0})$ we have that
\[
\left|f(\boldsymbol{y}) - f(\boldsymbol{x}) -
\boldsymbol{\nabla}f(\boldsymbol{x})\cdot(\boldsymbol{y} - \boldsymbol{x})\right| \le \epsilon \vert \boldsymbol{y} - \boldsymbol{x} \vert.
\]
Moreover, we say that $f$ is WLUD on $A$ if $f$ is WLUD at every point in $A$.
\end{defn}

Using Defintion \ref{def:wludn} and Definition \ref{defWLUDn1}, the natural way to define $k$ times weak local uniform differentiability (WLUD$^k$) at a point $\boldsymbol{x_0}$ or on an open subset $A$ of $\mathcal{N}^n$ is as follows.
\begin{defn}\label{defWLUDn1k}
Let $A\subset\mathcal{N}^n$ be open, let $f:A \to \mathcal{N}$, and let $\boldsymbol{x_0}\in A$ be given.  Then we say that $f$ is WLUD$^k$ at $\boldsymbol{x_0}$ if $f$ is $k$-times differentiable in a neighborhood
$\Omega$ of $\boldsymbol{x_0}$ in $A$ and if for every $\epsilon>0$ in $\mathcal{N}$ there exists $\delta>0$ in $\mathcal{N}$ such that $B_{\delta}(\boldsymbol{x_0})\subset\Omega$, and
for all $\boldsymbol{\xi},\boldsymbol{\eta}\in B_{\delta}(\boldsymbol{x_0})$ we have that
\[
\left|f(\boldsymbol{\eta})-f(\boldsymbol{\xi})-\sum_{j=1}^{k} \frac{1}{j!}\left[(\boldsymbol{\eta}-\boldsymbol{\xi})\cdot\nabla\right]^{j}f(\boldsymbol{\xi})\right| \le \epsilon \vert \boldsymbol{\eta} - \boldsymbol{\xi} \vert^k,
\]
where
\begin{eqnarray*}
\left[(\boldsymbol{\eta}-\boldsymbol{\xi})\cdot\nabla\right]^{j}f(\boldsymbol{\xi})&=&
\left.\left[(\eta_1-\xi_1)\frac{\partial}{\partial x_1}+\cdots+(\eta_n-\xi_n)\frac{\partial}{\partial x_n}\right]^{j}f(\boldsymbol{x})
\right|_{\boldsymbol{x}=\boldsymbol{\xi}}\\
&=&\sum_{l_{1},\ldots,l_{j}=1}^{n} \left(
\left.\frac{\partial^j f(\boldsymbol{x})}{\partial_{x_{l_{1}}}\cdots\partial_{x_{l_{j}}}}\right|_{\boldsymbol{x}=\boldsymbol{\xi}}
\prod_{m=1}^{j}\left(  \eta_{l_{m}}-\xi_{l_{m}}\right)  \right).
\end{eqnarray*}
Moreover, we say that $f$ is WLUD$^k$ on $A$ if $f$ is WLUD$^k$ at every point in $A$.
\end{defn}

\begin{defn}\label{defWLUDn1infty}
Let $A\subset\mathcal{N}^n$ be open, let $f:A \to \mathcal{N}$, and let $\boldsymbol{x_0}\in A$ be given.  Then we say that $f$ is WLUD$^\infty$ at $\boldsymbol{x_0}$ if $f$ is WLUD$^k$ at $\boldsymbol{x_0}$ for every
$k\in\mathbb{N}$. Moreover, we say that $f$ is WLUD$^\infty$  on $A$ if $f$ is WLUD$^\infty$ at every point in $A$.
\end{defn}

Now we are ready to state and prove the analog of Theorem \ref{thmtaylorseries1} for functions of $n$ variables.
\begin{theorem}\label{thmtaylorseriesn}
Let $A \subseteq \mathcal{N}^n$ be open, let $\boldsymbol{x_0}\in A$, and let $f : A \rightarrow \mathcal{N}$ be WLUD$^{\infty}$ at $\boldsymbol{x_0}$. For each $k\in\mathbb{N}$, let $\delta_k>0$ in $\mathcal{N}$ correspond to $\epsilon=1$ in Definition \ref{defWLUDn1k}. Assume that
\begin{eqnarray*}
\limsup_{{\tiny\begin{array}{l}j\rightarrow\infty\\l_1=1,\ldots,n\\
\vdots\\
l_j=1,\ldots,n
\end{array}}}\left(\frac{-\lambda\left(\left.\frac{\partial^j f(\boldsymbol{x})}{\partial_{x_{l_{1}}}\cdots\partial_{x_{l_{j}}}}\right|_{\boldsymbol{x}=\boldsymbol{x_0}}  \right)}{j}\right)&<&\infty\\
&\mbox{ }&\\
\text{ and }
\limsup _{k\rightarrow\infty}\lambda\left(\delta_k\right)&<&\infty.
\end{eqnarray*}
Then there exists a neighborhood $U$ of $\boldsymbol{x_0}$ in $A$ such that,  for any $\boldsymbol{\eta}\in U$, we have that
\[
f(\boldsymbol{\eta})=f(\boldsymbol{x_0})+\sum_{j=1}^{\infty} \frac{1}{j!}\left[(\boldsymbol{\eta}-\boldsymbol{x_0})\cdot\nabla\right]^{j}f(\boldsymbol{x_0}).
\]

\end{theorem}

\begin{proof}
Let
\[
\lambda_0=\limsup_{{\tiny\begin{array}{l}j\rightarrow\infty\\l_1=1,\ldots,n\\
\vdots\\
l_j=1,\ldots,n
\end{array}}}\left(\frac{-\lambda\left(\left.\frac{\partial^j f(\boldsymbol{x})}{\partial_{x_{l_{1}}}\cdots\partial_{x_{l_{j}}}}\right|_{\boldsymbol{x}=\boldsymbol{x_0}}  \right)}{j}\right).
\]
Then $\lambda_0\in\mathbb{R}$ and $\lambda_0<\infty$.

For all $k\in \mathbb{N}$, we have that $B_{\delta_k}(\boldsymbol{x_0})\subset A$, $f$ is $k$ times differentiable on $B_{\delta_k}(\boldsymbol{x_0})$, and
 \[
 \left|f(\boldsymbol{\eta})-f(\boldsymbol{x_0})-\sum_{j=1}^{k} \frac{1}{j!}\left[(\boldsymbol{\eta}-\boldsymbol{x_0})\cdot\nabla\right]^{j}f(\boldsymbol{x_0})\right| \le  \vert \boldsymbol{\eta} - \boldsymbol{x_0} \vert^k \text{ for all }\boldsymbol{\eta}\in B_{\delta_k}(\boldsymbol{x_0}).
\]
 Since $\limsup\limits _{k\rightarrow\infty}\lambda\left(\delta_k\right)<\infty$, there exists $t>0$ in $\mathbb{Q}$ such that $\limsup\limits _{k\rightarrow\infty}\lambda\left(\delta_k\right)<t<\infty$. Thus, there exists $N\in\mathbb{N}$ such that
 \begin{equation}\label{eqtaylorn:1}
  \lambda(\delta_k)<t\text{ for all }k>N.
  \end{equation}
  Let
  $\delta>0$ in $\mathcal{N}$ be such that $\lambda(\delta)>\max\{\lambda_0, t,0\}$. It follows from (\ref{eqtaylorn:1})
   that $\lambda(\delta)>\lambda(\delta_k)$ and hence $0<\delta\ll\delta_k$ for all $k>N$. Thus,
  $B_{\delta}(\boldsymbol{x_0})\subset A$, $f$ is infinitely often differentiable on $B_{\delta}(\boldsymbol{x_0})$, and
 \begin{equation}\label{eqtaylorn:2}
 \left|f(\boldsymbol{\eta})-f(\boldsymbol{x_0})-\sum_{j=1}^{k} \frac{1}{j!}\left[(\boldsymbol{\eta}-\boldsymbol{x_0})\cdot\nabla\right]^{j}f(\boldsymbol{x_0})\right| \le  \vert \boldsymbol{\eta} - \boldsymbol{x_0} \vert^k \ \forall \boldsymbol{\eta}\in B_{\delta}(\boldsymbol{x_0})\text{ and }\forall k>N.
\end{equation}

 Let $U=B_{\delta}(\boldsymbol{x_0})$; and let
$\boldsymbol{\eta}\in U$ be given.
Then we have that $\lambda(\vert \boldsymbol{\eta} - \boldsymbol{x_0} \vert)\ge \lambda(\delta)>\lambda_0$. We will show first that $\sum_{j=1}^{\infty} \frac{1}{j!}\left[(\boldsymbol{\eta}-\boldsymbol{x_0})\cdot\nabla\right]^{j}f(\boldsymbol{x_0})$ converges in $\mathcal{N}$. Since
$\lambda(\vert \boldsymbol{\eta} - \boldsymbol{x_0} \vert)>\lambda_0$, there exists $q>0$ in $\mathbb{Q}$ such that $\lambda(\vert \boldsymbol{\eta} - \boldsymbol{x_0} \vert)-q>\lambda_0$. Hence there exists $M\in\mathbb{N}$ such that
\[
\lambda(\vert \boldsymbol{\eta} - \boldsymbol{x_0} \vert)-q> \frac{-\lambda\left(\left.\frac{\partial^j f(\boldsymbol{x})}{\partial_{x_{l_{1}}}\cdots\partial_{x_{l_{j}}}}\right|_{\boldsymbol{x}=\boldsymbol{x_0}} \right)}{j}
\]
for all $j>M$ and for $l_1=1, \ldots, n$, $l_2=1, \ldots, n$, \ldots, $l_j=1, \ldots, n$. It follows that
\begin{eqnarray*}
\lambda\left(
\left.\frac{\partial^j f(\boldsymbol{x})}{\partial_{x_{l_{1}}}\cdots\partial_{x_{l_{j}}}}\right|_{\boldsymbol{x}=\boldsymbol{x_0}}
\prod_{m=1}^{j}\left(  \eta_{l_{m}}-x_{0,l_{m}}\right)  \right)
&\ge&\lambda \left(
\left.\frac{\partial^j f(\boldsymbol{x})}{\partial_{x_{l_{1}}}\cdots\partial_{x_{l_{j}}}}\right|_{\boldsymbol{x}=\boldsymbol{x_0}}
\vert \boldsymbol{\eta} - \boldsymbol{x_0} \vert^j\right)\\
&=&\lambda \left(
\left.\frac{\partial^j f(\boldsymbol{x})}{\partial_{x_{l_{1}}}\cdots\partial_{x_{l_{j}}}}\right|_{\boldsymbol{x}=\boldsymbol{x_0}}\right)+j\lambda\left(
\vert \boldsymbol{\eta} - \boldsymbol{x_0} \vert\right)\\
&>&jq
\end{eqnarray*}
for all $j>M$ and for $l_1=1, \ldots, n$, $l_2=1, \ldots, n$, \ldots, $l_j=1, \ldots, n$. Thus,
\begin{eqnarray*}
\lambda\left(\left[(\boldsymbol{\eta}-\boldsymbol{x_0})\cdot\nabla\right]^{j}f(\boldsymbol{x_0})\right)&=&
\lambda\left(\sum_{l_{1},\ldots,l_{j}=1}^{n} \left(
\left.\frac{\partial^j f(\boldsymbol{x})}{\partial_{x_{l_{1}}}\cdots\partial_{x_{l_{j}}}}\right|_{\boldsymbol{x}=\boldsymbol{x_0}}
\prod_{m=1}^{j}\left(  \eta_{l_{m}}-x_{0,l_{m}}\right)  \right)\right)\\
&>&jq
\end{eqnarray*}
for all $j>M$; and hence
\begin{eqnarray*}
\lim_{j\rightarrow\infty}\lambda\left(\frac1{j!}\left[(\boldsymbol{\eta}-\boldsymbol{x_0})\cdot\nabla\right]^{j}f(\boldsymbol{x_0})\right)&=&
\lim_{j\rightarrow\infty}\lambda\left(\left[(\boldsymbol{\eta}-\boldsymbol{x_0})\cdot\nabla\right]^{j}f(\boldsymbol{x_0})\right)\\
&\ge&q\lim_{j\rightarrow\infty}j=\infty.
\end{eqnarray*}
Thus,
\[
\lim_{j\rightarrow\infty}\left(\frac1{j!}\left[(\boldsymbol{\eta}-\boldsymbol{x_0})\cdot\nabla\right]^{j}f(\boldsymbol{x_0})\right)=0
\]
and hence $\sum_{j=1}^{\infty} \frac{1}{j!}\left[(\boldsymbol{\eta}-\boldsymbol{x_0})\cdot\nabla\right]^{j}f(\boldsymbol{x_0})$ converges in $\mathcal{N}$; that is,
\[
\lim\limits_{k\rightarrow\infty}\sum\limits_{j=1}^{k} \frac{1}{j!}\left[(\boldsymbol{\eta}-\boldsymbol{x_0})\cdot\nabla\right]^{j}f(\boldsymbol{x_0})\text{ exists in }\mathcal{N}.
\]

Taking the limit in (\ref{eqtaylorn:2}) as $k\rightarrow\infty$, we get:
\[
0\le\lim_{k\rightarrow\infty} \left|f(\boldsymbol{\eta})-f(\boldsymbol{x_0})-\sum_{j=1}^{k} \frac{1}{j!}\left[(\boldsymbol{\eta}-\boldsymbol{x_0})\cdot\nabla\right]^{j}f(\boldsymbol{x_0})\right|\le \lim_{k\rightarrow\infty} \vert \boldsymbol{\eta} - \boldsymbol{x_0} \vert^k,
\]
from which we obtain
\[
0\le \left|f(\boldsymbol{\eta})-f(\boldsymbol{x_0})-\lim_{k\rightarrow\infty}\sum_{j=1}^{k} \frac{1}{j!}\left[(\boldsymbol{\eta}-\boldsymbol{x_0})\cdot\nabla\right]^{j}f(\boldsymbol{x_0})\right|\le \lim_{k\rightarrow\infty} \vert \boldsymbol{\eta} - \boldsymbol{x_0} \vert^k.
\]
Since $\lambda(\vert\boldsymbol{\eta} - \boldsymbol{x_0}\vert)\ge \lambda(\delta)>0$, we obtain that
$\lim\limits_{k\rightarrow\infty} \left|\boldsymbol{\eta} - \boldsymbol{x_0}\right|^k=0$.
It follows that
\[
0\le\left|f(\boldsymbol{\eta})-f(\boldsymbol{x_0})-\sum_{j=1}^{\infty} \frac{1}{j!}\left[(\boldsymbol{\eta}-\boldsymbol{x_0})\cdot\nabla\right]^{j}f(\boldsymbol{x_0})\right|\le 0
\]
from which we infer that
\[
f(\boldsymbol{\eta})=f(\boldsymbol{x_0})+\sum_{j=1}^{\infty} \frac{1}{j!}\left[(\boldsymbol{\eta}-\boldsymbol{x_0})\cdot\nabla\right]^{j}f(\boldsymbol{x_0}).
\]
\end{proof}

%%\bibliographystyle{plain}
%%\bibliography{refs}

\end{document}